\documentclass [12pt,a4paper]{amsart}
\usepackage{amsmath}
\usepackage{cases}
\usepackage{amscd}
\usepackage{epsfig}
\usepackage{graphicx}
\usepackage{verbatim}
\usepackage{amssymb}
\usepackage{amsfonts}
\usepackage{amsbsy}
\usepackage{marginnote}
\usepackage{amsthm}
\usepackage{amstext}
\usepackage{amsopn}
\usepackage[dvips]{color}
\usepackage{indentfirst}
\usepackage{mathrsfs}

\newtheorem{thm}{Theorem}[section]
\newtheorem{defi}[thm]{Definition}
\newtheorem{cor}[thm]{Corollary}
\newtheorem{lem}[thm]{Lemma}
\newtheorem{prop}[thm]{Proposition}
\newtheorem{rem}[thm]{Remark}

\newtheorem{assm}[thm]{Assumption}

\def\<{\langle}    \def\>{\rangle}
\def\lp{\left(}  \def\rp{\right)}

\newcommand{\ve}{(V,E)}
\newcommand{\g}{(V,\omega,\mu)}

\newcommand{\lap}{\frac{1}{\mu(x)}\sum_{y\in V}\omega(x,y) }
\newcommand{\xt}{\lp \mathcal{X}_t\rp_{t\geq0} }

\def\NN{\mathbb{N}_{+}}
\def\RR{\mathbb{R}}

\def\ddd{\mathscr{D}}
\def\lll{\mathscr{L}}
\def\eee{\mathscr{E}}
\def\fff{\mathscr{F}}

\author[X. Huang]{Xueping Huang}
\address{Department of Mathematics\\ University of Bielefeld, 33501
Bielefeld , Germany}
\thanks{Research supported by Project CRC701}
\email{xhuang1@math.uni-bielefeld.de}
\title[Volume growth]{A note on the volume growth criterion for stochastic completeness of weighted graphs}
\def\func#1{\mathop{\rm #1}\nolimits}%
\def\FRAME#1#2#3#4#5#6#7#8
{
 \begin{figure}[ptbh]
 \begin{center}
 \includegraphics[height=#3]{#7}
 \caption{#5}
 \label{#6}
 \end{center}
 \end{figure}
}

\begin{document}
\numberwithin{equation}{section}

\begin{abstract}
We generalize the weak Omori-Yau maximum principle to the setting of strongly local Dirichlet forms. 
As an application, we obtain an analytic approach to compare the stochastic completeness of a weighted graph with that of an associated metric graph.
This comparison result played an essential role in the volume growth criterion of Folz \cite{FOLZSC}, who first proved it via a probabilistic approach. We also give an alternative analytic proof based on a criterion in Fukushima, Oshima, and Takeda \cite{FOT}. 
\end{abstract}

\date{\today}

\maketitle
\section*{Introduction and settings} 

In this article we develop two analytic approaches to relate the stochastic completeness (or conservativeness) of a weighted graph to that of an associated metric graph. These two approaches lead to analytic proofs of the volume growth criterion for weighted graphs recently obtained by Folz \cite{FOLZSC}. Both approaches involve only some simple elementary calculations.


 Our first approach is based on comparison of the existence of certain functions on a weighted graph and on a metric graph, via the weak Omori-Yau maximum principle in both settings. The main idea behind this approach is that the Dirichlet to Neumann problem for intervals can be used to relate a certain system of difference equations (inequalities) to a system of differential equations (inequalities). 

Our second approach is based on a useful general criterion of stochastic completeness (c.f. \cite{FOT}, Theorem 1.6.6). This approach involves several neat estimates and leads to an even simpler proof.

\subsection{Dirichlet forms}
The theory of Dirichlet forms offers a convenient common framework for diffusion type processes on metric graphs and jump type processes on weighted graphs, as required in our approaches. Let us first recall some basic facts and fix some notations, generally following the book of Fukushima et al. \cite{FOT}.

Without further specification, $\lp X, d \rp$ denotes a separable metric space such that all closed balls of the form $B_d(x,r)=\{y\in X: d(x, y)\le r\}$ are compact. Let $\mu$ be a Radon measure on $X$ with full support (i.e. $\func{supp}\mu=X$). Let
$\lp \eee, \fff \rp$ be a symmetric, regular Dirichlet form on the real Hilbert space $L^2 \lp X, \mu\rp$. Here $\fff$ is a dense subspace of
$\lp L^2 ( X, \mu), \<\cdot, \cdot\> \rp$, and is itself a Hilbert space with respect to the inner product $\eee_1 (\cdot,\cdot)=\eee(\cdot,\cdot) +\<\cdot, \cdot\>$. The regularity of $\lp \eee, \fff \rp$ means that $\fff\cap C_c(X)$ is dense in $\fff$ with $\eee_1$-norm and dense in $C_c(X)$ with uniform norm, where $C_c(X)$ is the space of compactly supported continuous functions on $\lp X, d\rp$. We denote $\fff_b =\fff\cap L^{\infty}(X, \mu)$, which forms an algebra. We will also make use of the spaces $\fff_c$, that is, the subspace of compactly supported functions in $\fff$, and $\fff_{b,c}=\fff_b\cap \fff_c$. 

It is a classical result of Beurling and Deny (see Theorem 3.2.1 in \cite{FOT}) that a general regular Dirichlet form allows a decomposition as
\begin{equation}
  \begin{aligned}
&\eee (u, v)=\eee^{(c)}(u,v)\\
&+\int_{X\times X-\func{diag}}\lp u(x)-u(y)\rp\lp v(x)-v(y)\rp J(dx, dy)+\int_X u(x)v(x)k(dx),
\end{aligned}
\end{equation}
for $u, v\in \fff\cap C_c(X)$. Here $\eee^{(c)}$ is the strongly local part of $\eee$, which satisfies that $\eee^{(c)}(u,v)=0$, for any pair of $u, v\in \fff\cap C_c(X)$ with $v$ being constant on a neighborhood of $\func{supp} u$. The integral with respect to the symmetric Radon measure $J$ on $X\times X$ off the diagonal is called the jump part of $\eee$. The last integral with $k$ a Radon measure on $X$ is the killing part. We always assume a vanishing killing part in this article.

There is a unique Hunt process $\lp \xt, (\mathbb{P}_x)_{x\in X}\rp$ (roughly speaking, a strong Markov process with all expected ``nice" properties) associated with a regular Dirichlet form. If a regular Dirichlet form has only the strongly local part non-vanishing, it is called strongly local and the corresponding process is a diffusion. For example, the Brownian motion on Euclidean spaces or Riemannian manifolds falls into this category. If the jump part is the only non-trivial part, the corresponding process is of pure jump type. Typical examples are $\alpha$-stable processes on Euclidean spaces and continuous time, reversible Markov chains with discrete state spaces.

There is a unique non-negative definite, self-adjoint operator $\lp \lll, \ddd \rp$, called the generator, naturally associated with $\lp \eee, \fff \rp$ in the following way (Corollary 1.3.1 in \cite{FOT}): 
\[\ddd\subseteq \fff; \eee(u, v)=\<\lll u, v\>, \forall u\in \ddd, v\in \fff.\]
Note that here we use the convention of signs that $\lll$ is non-negative definite, opposite to many authors.
Through functional calculus, we can define the heat semigroup $\lp P_t\rp_{t>0}$ on $L^2 ( X, \mu)$ as $P_t= \exp(-t\lll)$.
This is a Markovian semigroup and thus can be extended to $L^{\infty}(X, \mu)$ in a canonical way (see page 49 of \cite{FOT} for details).

\begin{defi}
  [Stochastic completeness] A Dirichlet form $\lp \eee, \fff \rp$ is called stochastically complete if the corresponding heat semigroup $\lp P_t\rp_{t>0}$ satisfies that
  \begin{equation}
    \label{eq-sc-defi}
    P_t \boldsymbol{1} =\boldsymbol{1} ~~~~ \mu\text{-a.e.~~}, \forall t>0,
  \end{equation}
  where $\boldsymbol{1}$ denotes the constant function taking value $1$. Otherwise, the Dirichlet form is called stochastically incomplete.
\end{defi}
 Sometimes we also talk about the stochastic completenss/incompletenss of the space $X$ when the Dirichlet form on it that we refer to is clear.
The following general criterion for stochastic completeness criterion will be applied in our second approach in Section \ref{sect-another-proof}.
\begin{thm}[Fukushima et al. \cite{FOT}]\label{thm-fot}
Let $(\eee, \fff)$ be a Dirichlet form on a $\sigma$-finite measure space $(X, d, \mu)$. Then it is stochastically complete if and only if there exists a sequence $\{v_n\}\subset \fff$ satisfying
\[0\le v_n\le 1, \lim_{n\rightarrow \infty}v_n =1 ~~~~\mu\text{-a.e.}\] 
such that
\[\lim_{n\rightarrow \infty}\eee(v_n, w)=0\]
holds for any $w\in \fff\cap L^1(X, \mu)$.
\end{thm}

In the probabilistic language, the heat semigroup $\lp P_t\rp_{t>0}$ corresponds to the transition semigroup of the corresponding Hunt process in the way that for any bounded Borel function $u$ on $X$,
  \[P_t u(x)= \mathbb{E}_x [u(\mathcal{X}_t)],~~~~ \mu\text{-a.e.~~~~} x, \forall t>0.\]
  Stochastic completeness has an intuitive interpretation that the Hunt process has infinite lifetime almost surely. We refer the reader to \cite{FOT} for details. The condition of stochastic completeness is useful in heat kernel estimates. See for example \cite{GriHu-inventiones}.

There is a large body of literature devoted to different types of criteria of stochastic completeness in various settings. In the following, we list several settings where this problem has been studied and make a brief summary of known results. In this article, we are mainly concerned with two types of criteria: function theoretic type and volume growth type. The first one relates stochastic completeness to the (non)existence of bounded solutions to certain equations or inequalities, while the second one relates it to the large scale geometry of the underlying metric space $(X, d)$.

\subsection{Riemannian manifolds}
The Brownian motion on a complete, connected, smooth Riemannian manifold $X$ can be naturally formulated in the language of Dirichlet forms. The metric measure space structure $\lp X, d, \mu\rp$ is given by the geodesic metric $d$ and the Riemannian volume $\mu$ on the manifold $X$. The Dirichlet form $\lp \eee, \fff \rp$ is given by (cf. \cite{GRIBOOK})
\[\fff= H^1_0(X),~~~~\eee(u, v)=\int_X \lp \nabla u \cdot \nabla v \rp d\mu, \forall u, v\in  H^1_0(X). \]
The generator $\lll$ is given by certain restriction of the Laplace-Beltrami operator $\Delta$ (non-negative definite sign convention).

One typical function theoretical criterion for stochastic completeness takes the following form.
\begin{thm}[Khas'minskii \cite{Khas}]\label{thm-khas}
  The manifold $X$ with the Dirichlet form $\lp \eee, \fff \rp$ is stochastically incomplete if and only if one of the following holds:
  \begin{enumerate}
    \item  for some/all $\lambda>0$, there is a bounded, nonzero, smooth solution to the equation
  \[\Delta u +\lambda u=0;\]

    \item   for some/all $\lambda>0$, there is a bounded, nonzero, non-negative $C^2$-solution to the inequality
  \[\Delta u +\lambda u\le 0.\]
  \end{enumerate}
\end{thm}
\begin{rem}\rm{
  The solution in (1) is called a $\lambda$-harmonic function, while the one in (2) is called a $\lambda$-subharmonic function.}
\end{rem}
The above theorem is largely due to Khas'minskii \cite{Khas}. See also the survey \cite{GRI-SURVEY} and the book \cite{GRIBOOK} of Grigor'yan.
A relaxed version of this Khas'sminski type criterion, the so-called weak Omori-Yau maximum principle, is first studied by Pigola, Rigoli and Setti \cite{PRSProc, PRSSurvey} (cf. related work of Omori \cite{Omori} and Yau \cite{Yau75}). In particular, they prove the following criterion for stochastic completeness.
\begin{thm}[Pigola et al. \cite{PRSProc}]\label{thm-weak-OY}
   The manifold $X$ is stochastically incomplete if and only if there is a $C^2$-function $u$ with $u^*=\sup u< \infty$  violating the weak Omori-Yau maximum principle, that is, $\exists \alpha>0$ such that
   \[\sup_{x\in\Omega_{\alpha}}\Delta u(x)\le -\alpha,\]
   where $\Omega_{\alpha}=\{x\in X, x>u^*-\alpha\}$.
\end{thm}
\begin{rem}\rm{
  The weak Omori-Yau maximum principle for a manifold amounts to the nonexistence of the function $u$ described in the above theorem, and is equivalent to the stochastic completeness of the manifold. The function $u$, if it exists, will be called a WOYMP-violating function following the notation of  B\"{a}r and Bessa \cite{Bar-Bessa}. In comparison with $\lambda$-harmonic functions, the WOYMP-violating functions are sometimes easier to work with. See for example \cite{Bar-Bessa} where it is applied to disprove a conjecture of Grigor'yan.}
\end{rem}

Many authors observed that stochastic completeness is closely related to the large scale geometry of the manifold. Grigor'yan \cite{GRI86-87, Gri88-89} obtained a sharp criterion merely in terms of the volume growth with respect to the geodesic metric.
\begin{thm}[Grigor'yan]\label{thm-gri-vol}
  If the volume of balls $B(x_0, r)$ centered at some fixed point $x_0\in X$ satisfies that
 \begin{equation}
\label{eq-vol}
\int^{\infty}\frac{rdr}{\log \lp\mu\lp B(x_0, r)\rp\rp}=\infty,
\end{equation}
then the manifold $X$ is stochastically complete.
\end{thm}
\begin{rem}\rm{
   The integral symbol $\int^{\infty}$ means that we only care about the possible divergence happening at $\infty$. An immediate consequence is that the volume growth with $\mu\lp B(x_0, r)\rp\le \exp \lp C r^2\rp$ for some constant $C>0$ implies stochastic completeness. This slightly weaker result can be obtained through various different approaches, see for example the works of Hsu \cite{HSU89}, Karp and Li \cite{Karp-Li}, Takeda \cite{Takeda89}, Davies \cite{DAV92}, and Pigola et al. \cite{PRS03}. It is worth pointing out that most authors work with the heat equation or the heat semigroup (``parabolic" approach), while in \cite{PRS03} Pigola et al. directly start from Theorem \ref{thm-weak-OY} (``elliptic" approach) to obtain the volume growth criterion.}
\end{rem}

\subsection{Strongly local Dirichlet forms}
Both Theorem \ref{thm-khas} and Theorem \ref{thm-gri-vol} are generalized to the setting of a strongly local, regular, irreducible Dirichlet form $\lp \eee, \fff \rp$ ($\eee=\eee^{(c)}$) by Sturm \cite{Sturm94}.

An analogue of the volume growth criterion in Theorem \ref{thm-gri-vol} cannot hold for the original metric $d$ on $X$ in general, since it does not have a quantitative relation with the Dirichlet form. Sturm adopted the notion of intrinsic metric (Biroli and Mosco \cite{BirMos91, BirMos95} and Davies \cite{DAV89}) to overcome this difficulty.

To introduce the intrinsic metric, we need some technical preparations following \cite{Sturm94} and Section 3.2 in \cite{FOT}. The results summarized here will also be used in Section \ref{sect-weak-OY}. For each $u\in \fff_b$, there exists a unique Radon measure $\Gamma(u, u)$, the so-called energy measure (or carr\'{e} du champ, cf. Le Jan \cite{LeJan}), such that
\begin{equation}
\label{eq-def-Gamma}
\int_X \phi d\Gamma(u,u)=\eee(u\phi, u)-\frac{1}{2}\eee(u^2, \phi), \forall \phi\in \fff\cap C_c(X).
\end{equation}
It can be proven that $\eee(u, u)=\Gamma(u,u)(X)$ for $u\in \fff_b$. By polarization, one can define $\Gamma(u, v)$ as
 \[\Gamma(u,v)=\frac{1}{2}\lp\Gamma(u+v,u+v)-\Gamma(u,u)-\Gamma(v,v)\rp,\]
 a signed Radon measure for $u, v\in\fff_b$, and we have $\eee(u, v)=\Gamma(u,v)(X)$. Note that $\Gamma(\cdot,\cdot)$ is symmetric.

 We say that a function $u\in \fff_{\func{loc}}$, the space of functions locally in $\fff$, if for any relatively compact open subset $G\subseteq X$, there exists a function $w$ in $\fff$ which coincides with $u$ $\mu$-a.e. on $G$. In a similar way we can define $\fff_{b, \func{loc}}$. It is easy to see that $\fff_{\func{loc}}\cap L^{\infty}\subseteq \fff_{b, \func{loc}}$. The bilinear, signed Radon measure valued map $\Gamma$ can be uniquely extended to $\fff\times\fff$ and then to $\fff_{\func{loc}}\times \fff_{\func{loc}}$.

  A typical example of energy measure is the one given by the gradient operator on Riemannian manifolds: $(\nabla u \cdot \nabla v) d\mu$ with $u, v\in H^1_0$. The general energy measure $\Gamma$ shares some nice properties with that on manifolds, for examples a Leibniz rule and a chain rule. More details will be given in Section \ref{sect-weak-OY}. The same as in \cite{Sturm94}, we will frequently take a quasi-continuous modification of a function in $\fff_{\func{loc}}$ without specification. 

A pseudo metric $\rho$ on $X$ can be defined using the energy measure $\Gamma$ by
\[\rho(x,y)=\sup \{u(x)-u(y): u\in \fff_{\func{loc}}\cap C(X), d\Gamma(u,u)\le d\mu\}.\]
On manifolds, the condition $d\Gamma(u,u)\le d\mu$ amounts to $\vert \nabla u\vert^2 \le 1$. In general, $\rho$ can be degenerate and Sturm proposed a basic assumption.
\begin{assm}\label{assm-basic}
  The pseudo metric $\rho$ is a metric and induces a topology that is equivalent to the original one.
\end{assm}
Under Assumption \ref{assm-basic}, Sturm proves that the volume growth condition (\ref{eq-vol}) with respect to the intrinsic metric $\rho$ implies the stochastic completeness of the Dirichlet form $\lp \eee, \fff \rp$ (Theorem 4 in \cite{Sturm94}).
A key technical feature of $\rho$ proved in \cite{Sturm94} is that the distance function $\rho_x=\rho(x, \cdot)$ satisfies
\begin{equation}
\label{eq-intrinsic}
\rho_x\in \fff_{\func{loc}}\cap C(X), \text{~~~~and~~~~} d\Gamma\lp \rho_x, \rho_x\rp \le d\mu,
\end{equation}
 for each $x\in X$. Sturm's result can be viewed as a generalization of Grigor'yan's, since on manifolds the intrinsic metric coincides with the geodesic metric.

Now we turn to Sturm's analogue of Theorem \ref{thm-khas}.  In such a general setting, it is hard to work with the generator directly, so a weak formulation of $\lambda$-subharmonic functions is adopted in \cite{Sturm94}.
\begin{defi}\label{defi-subharmonic}
  A non-negative function $u$ on $X$ is called $\lambda$-subharmonic if $u\in \fff_{\func{loc}}\cap L^{\infty}$ and
  \[\eee(u, \phi)+\lambda\int_X u\phi d\mu\le 0\]
  for all non-negative $\phi\in \fff_{c}$.
\end{defi}
\begin{rem}\rm{Note that we understand $\eee(u, \phi)$ as $\eee(u_{\phi}, \phi)$ for some $u_{\phi}\in \fff$ that coincides with $u$ on a relatively open set $G_{\phi}\supset \func{supp} \phi$. By strong locality, this is a coherent definition. We will adopt it later as well.

We can replace the condition ``$\phi\in \fff_{c}$" by ``$\phi\in \fff_{b, c}$".
Indeed, letting $\phi_n=\phi\wedge n$ for $n\in \NN$, we see that $\phi_n\in \fff_{b, c}$, $\func{supp}\phi_n\subseteq \func{supp}\phi$, and $\phi_n\rightarrow \phi$ both in the sense of $\eee_1$ and monotone convergence, as $n\rightarrow \infty$.
}
\end{rem}
Among other things, Sturm proved the equivalence of stochastic incompleteness and the existence of a non-negative, nonzero bounded $\lambda$-subharmonic function for some/all $\lambda>0$ (cf. Theorem 4 in \cite{Sturm94}).

Based on Sturm's result, in Section \ref{sect-weak-OY} we prove the following weak formulation (and generalization) of the weak Omori-Yau maximum principle in Theorem \ref{thm-weak-OY}. 
\begin{prop}\label{prop-weak-OY}
  Let $\lp \eee, \fff \rp$ be a strongly local, regular, irreducible Dirichlet form on $(X, d)$ satisfying Assumption \ref{assm-basic}. 
  Then  $\lp \eee, \fff \rp$ is stochastically incomplete if and only if there exist a non-negative function $v\in \fff_{\func{loc}}\cap L^{\infty}$ with $\func{esssup} v>0$, and two constants $\alpha, \delta>0$, such that
  \[\eee(v, \phi)\le -\delta\int_{\Omega_{\alpha}} \phi d\mu\]
  for all non-negative $\phi\in \fff_{b,c}$ with $\func{supp}\phi\subseteq \overline{\Omega_{\alpha}}$. Here $\Omega_{\alpha}=\{x\in X: v(x)>\func{esssup} v-\alpha\}$.
\end{prop}
\begin{rem}\rm{Before us, in the setting of metric graphs, a more direct analogue to Theorem \ref{thm-weak-OY} was obtained by Haeseler and Wojciechowski \cite{HaeWoj}. Proposition \ref{prop-weak-OY} is more suitable for our purpose since the weak formulation puts less restrictions on the regularity of the WOYMP-violating functions. The flexibility gained here may lead to further applications.
}
\end{rem}

\subsection{Weighted graphs}
Stochastic completeness of continuous time Markov chains, defined using the data of weighted graphs, was first studied extensively in probabilistic language. See for example the classical works of Chung \cite{Chung}, Feller \cite{Feller56, FEL66, Feller57}, Freedman \cite{freedman}, Reuter \cite{Reuter57} and the references therein.

In this article, we take the ``geometric analysis" point of view and treat weighted graphs as analogues of Riemannian manifolds. For this purpose, we choose to work with the language of Dirichlet forms. We basically follow the framework of Keller and Lenz \cite{KL} except that we restrict our attention to the locally finite and connected graph case.

Let $(V, E)$ be a locally finite, countably infinite, connected, undirected graph without loops or multi-edges. For abbreviation, we call such a graph a simple graph. Here $V$ is the vertex set, and $E$ the edge set which can be viewed as a symmetric subset of $V\times V$. We will write $x\sim y$ for a pair $(x,y)\in E$, and call them neighbors. Note that we do not allow loops, so $x\not\sim x$ for $x\in V$. For $x\in V$, we
define the degree of $x$ to be $\deg(x)=\#\{y\in
V: y\sim x\}$, i.e. the number of neighbors of $x$ which is finite by assumption. 

Connectivity of the graph guarantees a path of some length $n\in \mathbb{N}_{+}$ connecting any pair of distinct vertices $x, y\in V$, that is, a sequence of vertices $x_0, \cdots, x_n$ in $V$ such that
\[x_0=x, x_n=y, x_k\sim x_{k+1} \text{~~for all~~} 0\leq k\leq n-1.\]
The graph metric $d_{0}$ on the graph $(V, E)$, can then be defined through
\[d_{0}(x,y)=\inf\{n: \text{~~~~there exists a path of length~~}n\text{~~connecting~~}x,y\}.\]
for any pair of $x\neq y$. This induces the discrete topology on $V$. Note that $C_c(V)$ in this case is simply the space of functions that vanishes outside a finite set.

Set a weight function $\mu : V\rightarrow(0,\infty)$ on vertices, viewed as a Radon measure on $V$.
Let $\omega: V\times V\rightarrow[0,\infty)$ be a weight function on edges such that:
\begin{enumerate}
 \item $\omega(x,y)=\omega(y,x)$  for all  $x, y \in V$;
 \item $(x,y)\in E \Leftrightarrow \omega(x,y)>0$.
\end{enumerate}
The function $\omega$ plays the role of a jump kernel. The triple $(V, \omega, \mu)$ is usually called a weighted graph. Note that given a weighted graph $\g$, we can determine the edge set $E$ through $\omega$.

Define $(\eee, \fff_{\max})$ through:
\begin{align*}
\fff_{\max}=\{u: V\rightarrow\RR, \sum_x\sum_y \omega(x,y)\lp u(x)-u(y)\rp^2<\infty\};\\
\eee(u,u)=\frac{1}{2}\sum_x\sum_y \omega(x,y)\lp u(x)-u(y)\rp^2, \forall u\in \fff_{\max}.
\end{align*}
We have $C_c(V)\subseteq \fff_{\max}$ by locally finiteness. The Dirichlet form $(\eee, \fff)$ that we are interested in is the restriction of $(\eee, \fff_{\max})$ given by taking $\fff$
to be the closure of $C_c(V)$ in $\fff_{\max}$ with respect to the $\eee_1$-norm. In the probabilistic language, we are considering the ``minimal" Markov chain.

In \cite{KL}, the regularity of $(\eee, \fff)$ is shown, and the generator $\lll$ is determined to be a certain restriction of the so-called formal Laplacian (also denoted by $\Delta$), which takes the form:
\begin{equation}
\label{eq-formal-lap}
\Delta u(x)=\frac{1}{\mu(x)}\sum_{y}\omega(x,y)(u(x)-u(y)), \forall x\in V.
\end{equation}

The idea of studying this kind of operator as a discrete analogue of elliptic differential operators dates back at least to Courant, Friedrichs and Lewy \cite{CourFriedLewy}. For more recent works from this point of view, see for example Colin de Verdi\'{e}re \cite{CdVBook98}, Dodziuk \cite{Dod06}, Dodziuk and Mathai \cite{Dod-Mat06}, Wojciechowski \cite{WOJ-Indiana, WOJ-Survey}, and Weber \cite{WEBER}. 

An analogue of Theorem \ref{thm-khas} is proved by Keller and Lenz \cite{KL} in this setting.
\begin{thm}[\cite{KL}]
  \label{thm-KL} The weighted graph $X$ (not necessarily connected) with the Dirichlet form $\lp \eee, \fff \rp$ is stochastically incomplete if and only if for some/all $\lambda>0$, there is a bounded, nonzero, non-negative solution to the equation
$\Delta u +\lambda u=0$.
\end{thm}
Based on this result, there is also an analogue of Theorem \ref{thm-weak-OY} for weighted graphs.
\begin{thm}[\cite{HXP-OY}]\label{thm-weak-OY-graph}
Let $\g$ be a simple weighted graph with the formal Laplacian $\Delta$ as defined by (\ref{eq-formal-lap}). Then it is stochastically incomplete if and only if there is a non-negative WOYMP-violating function $u$ on $V$ with $u^*=\sup u< \infty$, that is, $\exists \alpha>0$ such that
   \[\sup_{x\in\Omega_{\alpha}}\Delta u(x)\le -\alpha,\]
   where $\Omega_{\alpha}=\{x\in X, x>u^*-\alpha\}$.
\end{thm}
\begin{rem}
  \rm{Theorems \ref{thm-fot} and \ref{thm-weak-OY-graph} can be applied to prove certain stability results, cf. \cite{Gri-H-Mas} for the former, and \cite{HXP-OY} for the later.
Indeed in Theorem \ref{thm-weak-OY-graph}, we do not need locally finiteness of the graph, if we treat the domain of the formal Laplacian more carefully (see \cite{HXP-OY} for details).
  }
\end{rem}

The full analogue of Theorem \ref{thm-gri-vol}, however is obtained only very recently by Folz \cite{FOLZSC}. There are two major difficulties to obtain this type of volume growth criterion for stochastic completeness of weighted graphs. The first one is that the graph metric is not suitable to get a meaningful volume growth criterion. This issue was first shown by Wojciechowski \cite{WOJ-Survey} through his example of ``anti-trees". Similar to the general strongly local case, we need a metric that plays the role of the intrinsic metric.

\begin{defi}
  \label{defi-adapted-distance-new}
  We call a metric $d$ on
a simple weighted graph $(V, \omega,\mu)$ weakly adapted if for some $c_0>0$,
\begin{equation}\label{eq-adapted-condition}
\lap \lp d(x,y)\wedge c_0\rp^2\leq 1
\end{equation}
for every $x\in V$. Such a metric called adapted if in addition it satisfies
\[\omega(x,y)>0\Rightarrow d(x,y)\le c_0.\]
\end{defi}

\begin{rem}
  \rm{
   Note that (\ref{eq-adapted-condition}) is an analogue of (\ref{eq-intrinsic}). 
   There always exist an adapted metric (non-unique in general) on a simple weighted graph. For a more detailed study of adapted metrics in relation with essential self-adjointness, we refer to \cite{HKMW}. Many examples are given there. 

 A direct adaption of the definition of intrinsic metric in the strongly local case does not work, as discussed in \cite{FOLZSC}. In \cite{Huang-unique} and \cite{Gri-H-Mas} the notion of adapted metrics, was applied to obtain some weaker volume growth criterion. These works are inspired by the intrinsic metric for general Dirichlet forms developed by Frank, Lenz and Wingert \cite{FLW}, as well as the work of Masamune and Uemura \cite{MU}. See also Folz \cite{Folz} for related independent work in the context of heat kernel estimates on weighted graphs.
 }
\end{rem}

The second difficulty is that in the discrete setting, there is no analogue of the chain rule. So the geometric analytic methods applied in \cite{Huang-unique} only yields a volume growth criterion with respect to an adapted metric of the type $\mu\lp B(x_0, r)\rp\le \exp\lp c r\log r\rp$ with $0<c<\frac{1}{2}$. A similar volume growth criterion was shown in \cite{Gri-H-Mas} for general jump type Dirichlet forms. See Masamune, Uemura and Wang \cite{MUW} and Shiozawa \cite{Shiozawa} for improvements and further developments.

Folz \cite{FOLZSC} made the breakthrough by utilizing the so-called metric graphs. Roughly speaking, metric graphs are graphs enriched with intervals attached to edges, offering a continuous version of weighted graphs. More details will be given soon in the next subsection.
Folz's result can be stated as follows: 
\begin{thm}[Folz]
  \label{thm-Folz} Let $\g$ be a simple weighted graph. Let $d$ be a weakly
  adapted metric such that all closed metric balls $B_d (x, r)$ are finite.
  If the volume growth with respect to $d$ satisfies:
  \begin{equation}
\label{eq-vol-graph}
\int^{\infty}\frac{rdr}{\log \lp\mu\lp B_d(x_0, r)\rp\rp}=\infty,
\end{equation}
for some reference point $x_0\in V$, then the corresponding Dirichlet form $(\eee, \fff)$ is stochastically complete.
\end{thm}
\begin{rem}
  For the finiteness of metric balls in $d$, a typical sufficient condition is that all balls have finite measure and there is a positive lower bound for the function $\mu$. This is the assumption adopted in \cite{Huang-unique}.
\end{rem}

We end this subsection with an important remark about weakly adapted metrics. For a simple weighted graph $\g$ with a weakly adapted metric $d$, we can define a new weighted graph $(V,\omega',\mu)$ by
\[\omega'(x,y)=\omega(x,y)\text{~~~~if~~~~}d(x,y)\le c_0; ~~\omega'(x,y)=0\text{~~~~if~~~~}d(x,y)> c_0.\] 
It is direct to see that $d$ is an adapted metric for $(V,\omega',\mu)$. Also note that both weighted graphs share the same metric measure space structure $(V, d,\mu)$. By a stability result for general jump processes in \cite{Gri-H-Mas} (Theorem 2.2), if all metric balls in $(V,d)$ are finite, then $\g$ is stochastically complete if and only if $(V,\omega',\mu)$ is. The only subtle issue is that the new weighted graph $(V,\omega',\mu)$ may be not connected. By Theorem \ref{thm-KL}, it is stochastically complete if and only if each connected component is. If the volume growth condition (\ref{eq-vol-graph}) holds for $\g$, it obviously holds for each of the connected components of $(V,\omega',\mu)$.
Thus, to prove Theorem \ref{thm-Folz}, it suffices to work with the case that $d$ on $\g$ is an adapted metric. Later on, we only work with adapted metrics.

\subsection{Metric graphs}
Here we collect some basic facts that we will apply later. For more complete expositions of the theory of metric graphs (or quantum graphs), we refer to the surveys of Kuchment \cite{KuchSurveyI, KuchSurveyII}. See also the work of Haeseler \cite{Hae-heat-kernel}.
For the theory of Sobolev spaces involved, we closely follow Brezis \cite{BREbook}. We do not claim any originality for the results of this subsection. However, some care has to be taken since our setting does not strictly fit the framework in \cite{KuchSurveyI, KuchSurveyII}.

Let $\ve$ be a simple graph as before. We enrich the graph $\ve$ with the following data:
\begin{enumerate}
  \item a map $\tau:E\rightarrow\{1,-1\}$, satisfying $\tau\lp (x,y)\rp=-\tau \lp (y,x)\rp$ for all $(x,y)\in E$, $E_+:=\tau^{-1}(\{1\})$;
  \item three positive functions $l$, $p$, $q$ defined on $E_+$;
  \item a family of marked intervals $\{I(e)\}_{e\in E_+}$, where $I(e)=[0, l(e)]\times\{e\}$.
\end{enumerate}
The map $\tau$ in $(1)$ above defines an orientation on $\ve$. Denote $E_+(x)=\{(x,y): (x,y)\in E_+\}$, $E_-(x)=\{(y,x): (x,y)\in E_+\}$ and $E_x=E_+(x)\cup E_-(x)$. 
For $e=(x,y)\in E_+$, we call $x$ the source $s(e)$ of $e$ and $y$ the target $t(e)$. The metric graph $X$ is then given by the quotient of the disjoint union
$\bigsqcup_{e\in E_+ }I(e)$
through the map $\pi$ such that for each $x\in V$, those marked points $0_e$ with $e\in E_+(x)$ and $l(e)_e$ with $e\in E_-(x)$ are identified. 


The vertex set $V$ can be viewed as a subset of $X$ through $\pi$. Sometimes, we do not distinguish a marked interval with its image under $\pi$, as a subset of the metric graph $X$, and simply call it an edge of the metric graph.

 We equip the marked intervals with the natural Euclidean metric on them, as well as the measure $q(e) dm(e)$ for $e\in E_+ $ where $dm(e)$ is the Lebesgue measure on $[0, l(e)]$. The topology on the metric graph $X$ is the natural quotient topology. There is also a natural shortest path metric $d_{l}$ on it, that is, the quotient metric induced by the marked intervals under $\pi$. 
 The natural quotient measure $\mu$ on $X$ is defined as the push-forward measure of $\oplus_{e\in E_+}q(e)m(e)$ under the map $\pi$.

On a marked interval $I(e)$, we equip the Sobolev space $W^{1,2}(e)=W^{1,2}((0, l(e)))$ with the norm
  \[\parallel f\parallel_{W^{1,2}(e)}^2=q(e)\int_{0}^{l(e)}f^2 dm(e)+ p(e)\int_{0}^{l(e)}(f')^2 dm(e).\]
Note that $f'$ denotes the weak derivative of $f\in W^{1,2}((0, l(e))) $ that is defined almost everywhere. Later on, for a function in $W^{1,2}(e)$, we always mean its continuous representative on $[0, l(e)]$. Consider the Sobolev space   
\[W^{1,2}(X)=\{u\in C(X)\cap \bigoplus_{e\in E_+}W^{1,2}(e): \sum_{e\in E_+}\parallel u\vert_{I(e)}\parallel_{W^{1,2}(e)}^2<\infty\}.\]
A natural bilinear form $\eee$ can then be defined on $W^{1,2}(X)$ as
\[\eee(u,u)=\sum_{e\in E_+}p(e)\int_{0}^{l(e)}(u'\vert_{I(e)})^2 dm(e).\] 

To go further, we need make the following basic topological assumption.
\begin{assm}
  \label{assm-metric-graph}All closed balls in $(V ,d_{l})$ (as a subspace of $(X, d_{l})$) are finite. 
\end{assm}
\begin{rem}
  \rm{Note the similarity with the topological assumption in Theorem \ref{thm-Folz}. Later in Lemma \ref{lem-comparison}, we will see the relation between these two assumptions.
  }
\end{rem}
Since $(V, E)$ is locally finite, Assumption \ref{assm-metric-graph} implies that each metric ball has nonempty intersection with only finitely many marked intervals. By Sobolev embedding (cf. Theorem 8.8 in \cite{BREbook}), we see that a Cauchy sequence in $W^{1,2}(X)$ converges uniformly on each compact subset of $X$, and thus the limit is again a continuous function. This observation leads to the fact that $(\eee, W^{1,2}(X))$ is a closed Dirichlet form. We do not know how to show this without Assumption \ref{assm-metric-graph}. 

Let $C_{\func{Lip}, c} (X)=C_{\func{Lip}}(X)\cap C_c(X)$ be the space of compactly supported Lipschitz functions on $(X, d_{l})$. By Assumption \ref{assm-metric-graph}, such a function is $\mu$-a.e. differentiable and the derivative is essentially bounded. Thus we have $C_{\func{Lip}, c} (X)\subseteq W^{1,2}(X)$.
We can then define the space $\fff$ as the closure of $C_{\func{Lip}, c} (X)$ with respect to $\eee_1$-norm. The form $(\eee, \fff)$, as a restriction of $(\eee, W^{1,2}(X))$, is then a regular Dirchlet form (note that $C_{\func{Lip}, c} (X)$ is dense in $C_c(X)$ with uniform norm). From the topology, we also see that the space $C_{\func{Lip}, \func{loc}} (X)$ of locally Lipschitz functions is a subspace of $\fff_{\func{loc}}$. Indeed, for a locally Lipschitz function $u$ and two relative open sets $G, H$ with $\overline{G}\subset H\subset X$, we can always extend $u\vert_G$ by linear functions to get a Lipschitz function supported in $\overline{H}$, since only finitely many marked intervals are involved in this procedure.  

Everything well fits into the framework of strongly local Dirichlet forms. For simplicity, we now assume that $p(e)=q(e)$ for all $e\in E_+$ (in our application this particular choice holds). The energy measure $\Gamma$ in this setting has the simple form
\[d\Gamma (u, u)\vert_{I(e)}=p(e)(u'\vert_{I(e)})^2 dm(e),\]
for each $I(e)$ and $u\in \fff_{\func{loc}}$. The condition $d\Gamma (u, u)\le d\mu$ then simply reads $\vert u'\vert \le 1$ for $u\in \fff_{\func{loc}}$, and such a function $u$ is in $\fff_{\func{loc}}\cap C_{\func{Lip}}(X)=C_{\func{Lip}}(X)$.
The intrinsic metric $\rho$ can then be directly shown to coincide with $d_{l}$ \cite{FOLZSC}. 

Let $\lp V, \omega, \mu\rp$ be a simple weighted graph and $d$ an adapted metric on $V$ such that all metric balls in $(V, d)$ are finite.
We can construct a metric graph $X$ using the data $\g$ and $d$, by setting
the functions $l, p, q$ to be
 \[l(e)=d(x,y), p(e)=q(e)=\omega(x,y)d(x,y),\] 
for $e=(x,y)\in E_+$. 

The proof of Theorem \ref{thm-Folz} can then be split into two steps.
First, 
the volume growth of the metric graph $X$ with respect to $d_l$ is no larger than that of the weighted graph $\g$ with respect to $d$. This will be explicitly formulated as Lemma \ref{lem-comparison}.

The second step is a comparison of stochastic completeness of $X$ and $\g$.
\begin{prop}\label{prop-comparison}
  Let $\lp V, \omega, \mu\rp$ be a simple weighted graph and $d$ an adapted metric on $V$ such that all metric balls in $(V, d)$ are finite. Let $X$ be the associated metric graph constructed as above. If $X$ is stochastically complete, then so is the weighted graph $\g$.
\end{prop}

Two different proofs of Proposition \ref{prop-comparison} will be given in Section \ref{sect-proof-vol} and Section \ref{sect-another-proof}. 

For metric graphs, as we already mentioned, analogues of the weak Omori-Yau maximum principle and Khas'minskii type results are first due to Haeseler and Wojciechowski \cite{HaeWoj}. Another remark is that the Dirichlet form $(\eee, \fff)$ constructed here is always irreducible, due to the connectivity of the metric graph $X$.

\subsection{Summary and organization of the article}
In \cite{FOLZSC}, Folz first proposed the idea of relating stochastic completeness of weighted graphs to that of metric graphs with loops, which leads to his proof of Theorem \ref{thm-Folz}. Roughly speaking, for each weighted graph with an adapted metric on it, he constructed a corresponding metric graph, and then made comparison between the Markov chain on graph and diffusion on metric graph of the holding times and jumping probabilities. His comparison relies on some quite nontrivial probabilistic calculations. The stochastic completeness of the weighted graph then follows from that of the corresponding metric graph, where the volume growth criterion of Sturm applies.

In this article, we also associate a metric graph $X$ to a simple weighted graph $\g$ with an adapted metric $d$, and make a reduction to Sturm's volume growth criterion. However, we develop two purely analytic approaches to prove the key comparison result Proposition \ref{prop-comparison}. The first approach is based on the weak Omori-Yau type maximum principles: Theorem \ref{thm-weak-OY-graph} for weighted graphs and Proposition \ref{prop-weak-OY} for metric graphs. When $\g$ is stochastically incomplete with a WOYMP-violating function $u$, we construct a WOYMP-violating function $v$ on $X$ using $u$, showing the stochastic incompleteness of $X$. The functions $u$ and $v$ are related through the Dirichlet to Neumann problem for intervals.
The second approach is based on the general criterion Theorem \ref{thm-fot}. Assuming the stochastic completeness of the metric graph $X$,
a candidate sequence $\{v_n\}$ as required in Theorem \ref{thm-fot} for the weighted graph $\g$ is naturally
given by the restriction of the sequence $\{\hat{v}_n\}$
 on $X$ to $V$. Checking that this sequence $\{v_n\}$ indeed works involves some simple but interesting calculations.
In some sense, the second proof is even simpler than the first one. However,
the application of Proposition \ref{prop-weak-OY} in the first approach seems more conceptual and the Dirichlet to Neumann problem
relating graphs and metric graphs
may be worth further developing.

We organize our paper as follows. In Section \ref{sect-weak-OY}, we give a proof of Proposition \ref{prop-weak-OY}. Section \ref{sect-metric-graph} is devoted to the comparison of volume growth of a weighted graph with the associated metric graph. We construct the WOYMP-violating function and complete the first analytic proof of Proposition \ref{prop-comparison}
in the Section \ref{sect-proof-vol}.  The second proof is presented in the last section.
















\section{Weak Omori-Yau maximum principle}\label{sect-weak-OY}
In this section we give a proof of Proposition \ref{prop-weak-OY}.
We start with some basic facts about a strongly local, regular Dirichlet form $\lp \eee, \fff \rp$ with energy measure $\Gamma$ on a metric space $\lp X, d\rp$. 

\begin{lem}[\cite{FOT}, the Leibniz rule and chain rule]\label{lem-prod-chain}
Let $u, v, w\in \fff_{b, \func{loc}}$ (with some quasi-continuous modifications chosen) and $f\in C^1 (\mathbb{R})$. Then we have 
\begin{align}
\label{eq-prod}
d\Gamma(uv, w)&=u d\Gamma(v, w)+v d\Gamma(u, w);\\
\label{eq-chain}
d\Gamma(f(u), v)&=f'(u) d\Gamma(u, v).
\end{align}
 \end{lem}
An immediate consequence of the Leibniz rule (\ref{eq-prod}) is the following extension of (\ref{eq-def-Gamma}) implicitly contained in \cite{FOT}.
\begin{cor}[\cite{FOT}]
\label{cor-formula}
For all $u, v, w\in \fff_{b}$, 
  \begin{equation}
\label{eq-Gamma-eee}
\int_X w d\Gamma (u,v)=\frac{1}{2}\lp \eee(uw, v)+\eee(u, vw)-\eee(uv, w)\rp.
\end{equation}
\end{cor}
\begin{proof}
  For $w\in \fff\cap C_c(X)$, this is simply the polarization of (\ref{eq-def-Gamma}). For the general case, consider permutations of (\ref{eq-prod}), we have
  \begin{align*}
\eee(uv, w)=\int_X d\Gamma(uv, w)&=\int_X u d\Gamma(v, w)+\int_X v d\Gamma(u, w),\\
\eee(uw, v)=\int_X d\Gamma(uw, v)&=\int_X u d\Gamma(v, w)+\int_X w d\Gamma(u, v),\\
\eee(u, vw)=\int_X d\Gamma(vw, u)&=\int_X v d\Gamma(u, w)+\int_X w d\Gamma(u, v),
\end{align*}
and the assertion follows easily.
\end{proof}
\begin{lem}\label{lem-technical-oy}
Let $u\in \fff_{b}$ (with some quasi-continuous modification chosen), $\phi\in\fff_{b,c}$ and $f\in C^3 (\mathbb{R})$. Then we have
\[\eee(f(u),\phi)=\eee(u, f'(u)\phi)-\int_X f''(u)\phi d\Gamma(u,u).\]  
In particular, if $f''\ge 0$ on $\RR$ and $\phi\ge 0$, then
\[\eee(f(u),\phi)\le\eee(u, f'(u)\phi).\]
 \end{lem}
\begin{proof}
  Note first that $f(u), f'(u), f''(u)\in \fff_{b}$ since $u\in \fff_{b}, f\in C^3 (\mathbb{R})$. 
 First, we have 
 \begin{align*}
\eee(f(u), \phi)&=\int_X d\Gamma(f(u),\phi)\\
&\overset{(\ref{eq-chain})}{=}\int_X f'(u) d\Gamma(u,\phi)\\
&\overset{(\ref{eq-Gamma-eee})}{=}\frac{1}{2}\lp \eee(f'(u)u, \phi)+ \eee(u, f'(u)\phi)-\eee(f'(u), u\phi)\rp.
\end{align*}
For the term involving $f''$, we have
\begin{align*}
\int_X f''(u)\phi d\Gamma(u,u)&\overset{(\ref{eq-chain})}{=}\int_X \phi d\Gamma(f'(u),u)\\
&\overset{(\ref{eq-Gamma-eee})}{=}\frac{1}{2}\lp \eee(u, f'(u)\phi)+\eee(f'(u), u\phi)-\eee(f'(u) u, \phi)\rp.
\end{align*}
The assertion follows by summing up the above two identities.
\end{proof}
\begin{rem}
  \rm{This is an argument similar to the proof of Lemma 2 in \cite{Sturm94}.
  }
\end{rem}
Now we have all the tools to prove Proposition \ref{prop-weak-OY}.
\begin{proof}[Proof of Proposition \ref{prop-weak-OY}]
By Sturm's result, it amounts to proving that there exists a nonzero non-negative bounded $\lambda$-subharmonic function $u$ (as in Definition \ref{defi-subharmonic}) if and only if there exists a WOYMP-violating function $v$ (as described in Proposition \ref{prop-weak-OY}).

The ``only if" part is easier. We simply prove that $v=u$ is a WOYMP-violating function. Since $u$ is non-negative, nonzero, bounded, we can let $\alpha=\frac{1}{2}\func{esssup} u>0$ and
$\delta=\frac{1}{2}\lambda\func{esssup} u>0$. The WOYMP-violating function satisfies
$u\ge\frac{1}{2}\func{esssup} u$, $\mu$-a.e. on  $\Omega_{\alpha}$.

Then for all non-negative $\phi\in\fff_c$ (no need to assume that $\func{supp} \phi\subseteq \overline{\Omega}_{\alpha}$), we have
\[\eee(u, \phi)\le -\lambda\int_X u\phi d\mu\le-\lambda\int_{\Omega_{\alpha}} u\phi d\mu\le -\delta\int_{\Omega_{\alpha}} \phi d\mu.\] 

Now we turn to the ``if" part. Without loss of generality, we can assume that $c=\func{esssup} v-\alpha>0$ ($\Omega_{\alpha}$ shrinks as $\alpha$ getting smaller). Define a function $f\in C(\RR)$ by $f(x)=\lp x-c\rp_+$. Let $u=f(v)$, which is in $\fff_{\func{loc}}\cap L^{\infty}$ by regularity and Markov property. 
Let $\lambda=\frac{\delta}{2\alpha}$. What we need show is that for all non-negative $\phi\in \fff_{b,c}$,
\[\eee(u, \phi)+\lambda\int_X u\phi d\mu=\eee(u, \phi)+\frac{\delta}{2\alpha}\int_X u\phi d\mu\le 0.\]
Note that $0\le u\le \alpha$ and $u=0$ on $\Omega_{\alpha}^c$, $\mu$-a.e., whence it suffices to prove
\[\eee(u, \phi)\le -\delta\int_{\Omega_{\alpha}}\phi d\mu.\]
We fix a non-negative $\phi\in \fff_{b,c}$. Choose a relatively compact open set $G_{\phi}\supset \func{supp}\phi$ and a function $\tilde{v}\in \fff_b$ that coincides with $v$ $\mu$-a.e. on $G_{\phi}$. It follows that the function $f(\tilde{v})$  coincides with $u=f(v)$ $\mu$-a.e. on $G_{\phi}$. Hence by the strong locality of the energy measure $\Gamma$,
\[\eee(\tilde{v}, \phi)=\eee(v, \phi);~~~~\eee(f(\tilde{v}), \phi)=\eee(u, \phi).\]

Through smooth mollifiers, it is easy to construct a sequence of smooth functions $\{f_n\}_{n\in \NN}$ on $\RR$ such that the following conditions hold,
\begin{enumerate}
  \item $f_n (x)=0$ for $x\le c+\frac{1}{n+2}$;
  \item $f_n (x)=x-(c+\frac{1}{n+1})$ for $x\ge c+\frac{1}{n}$;
  \item $f_n (x)>(x-(c+\frac{1}{n+1}))_+$, $f_n'(x)\in (0,1)$,

\noindent and $f_n''(x)>0$ for $x\in (c+\frac{1}{n+2},c+\frac{1}{n})$.
\end{enumerate} 
Note that $\{f_n\}_{n\in \NN}$ is monotone since for any $x\in \RR$,
\[f_{n+1} (x)\ge(x-(c+\frac{1}{n+2}))_+\ge f_n (x).\]

Considering the definition of $\tilde{v}$, we see that for each $n$, $\phi f_n'(\tilde{v})\in \fff_{b,c}$, and
\begin{equation}
\label{eq-supp-phi-f'}
\func{supp} \phi f_n'(\tilde{v})\subseteq G_{\phi}\cap\overline{\Omega_{\alpha}}.
\end{equation}
By Lemma \ref{lem-technical-oy}, we have for each $n\in \NN$,
\begin{align*}
\eee(f_n(\tilde{v}),\phi)&\le\eee(\tilde{v}, f_n'(\tilde{v})\phi)\\
&\overset{(\ref{eq-supp-phi-f'})}{=}\eee(v, f_n'(\tilde{v})\phi)\\
&\overset{(\ref{eq-supp-phi-f'})}{\le}-\delta\int_{\Omega_{\alpha}}f_n'(\tilde{v})\phi d\mu\\ 
&\le -\delta\int_{\Omega_{\alpha-\frac{1}{n}}}\phi d\mu.
\end{align*}
It is direct to see that
\[\lim_{n\rightarrow\infty}\int_{\Omega_{\alpha-\frac{1}{n}}}\phi d\mu=\int_{\Omega_{\alpha}}\phi d\mu.\]

We are left to prove that
\[\lim_{n\rightarrow\infty}\eee(f_n(\tilde{v}),\phi)=\eee(f(\tilde{v}),\phi),\]
while the latter is equal to $\eee(u,\phi)$.
This follows from the observation that there exists a sequence of normal contractions $\{g_n\}_{n\in \NN}$ on $\RR$ such that $f_n= g_n(f)$ and
$g_n\rightarrow \func{Id}$ uniformly.
A standard argument (cf. the proof of Theorem 1.4.2 (iii) in \cite{FOT}) gives that $f_n(\tilde{v})\rightarrow f(\tilde{v})$ with $\eee_1$-norm as $n\rightarrow \infty$, and this finishes the proof.
\end{proof}


\section{Construction of the metric graph}\label{sect-metric-graph}
Let $\lp V, \omega, \mu\rp$ be a simple weighted graph and $d$ an adapted metric on $V$ such that all metric balls in $(V, d)$ are finite. 

Now we construct a metric graph using the data $\g$ and $d$. 
We start from the graph structure $\ve$ of $\g$. Following the structure of a metric graph as in subsection 0.5, we fix 
the functions $l, p, q$ by
 \[l(e)=d(x,y), p(e)=q(e)=\omega(x,y)d(x,y),\] 
for $e=(x,y)\in E_+$. 

To avoid confusion, from now on, we denote the measure on $X$ constructed using $q$ by $\hat{\mu}$. Denote the natural quotient metric on $X$ by $d_{l}$ as before. The following type comparison lemma is due to Folz \cite{FOLZSC}. We include a proof here for completeness.
\begin{lem}
  \label{lem-comparison}
  We have the following comparisons:
  \begin{enumerate}
\item  for each pair of points $x,y\in V$, we have $d(x,y)\le d_{l}(x,y)$, where for $d_{l}(x,y)$, $x,y$ are viewed as points in $X$;
\item for all $r>0$, $x_0\in V$, we have
\[\hat{\mu}\lp B_{d_{l}}^{X}(x_0, r)\rp \le \mu\lp B_{d}^{V}(x_0, r)\rp,\]
where $B_{d_{l}}^{X}(x_0, r)$ is understood as a subset of $(X, d_{l})$, and $B_{d}^{V}(x_0, r)$ is a subset of $V$.
  \end{enumerate}
\end{lem}
\begin{proof}
  (1) For distinct $x, y\in V\subset X$, by definition of $d_{l}$ as the quotient metric, we have
  \[d_{l} (x,y)=\inf \{\sum_{i=0}^{n-1}d(x_i, x_{i+1}): \{x_i\}_{i=0}^{n} \text{~~is a path connecting~~}x,y\}.\]
  By triangle inequality for $d$, the assertion follows.

  (2) Note that for $e=(x,y)\in E_+$, we have $\hat{\mu}(I(e))=q(e)l(e)=\omega(x,y)d^2(x,y)$. 

  For each $x\in V\subset X$, the total measure of the set $E_x$ of marked intervals with $x$ as a vertex, can be calculated as
  \[\sum_{e\in E_x}\hat{\mu}(I(e))\le \sum_{y, y\sim x}\omega(x,y)d^2(x,y)\le \mu(x).\]
  The second inequality follows from the adapted-ness condition (\ref{eq-adapted-condition}).

Fix some $r>0$ and $x_0\in V$. By (1), we have $B_{d_{l}}^{X}(x_0, r)\cap V\subseteq B_{d}^{V}(x_0, r)$. Let $A\subset E_+$ be such that the corresponding marked intervals have nonempty intersection with $B_{d_{l}}^{X}(x_0, r)$.
   Then we see that
   \[\hat{\mu}\lp B_{d_{l}}^{X}(x_0, r)\rp\le \sum_{e\in A}\hat{\mu}(I(e))\le\sum_{x\in B_{d_{l}}^{X}(x_0, r)\cap V}\sum_{e\in E_x}\hat{\mu}(I(e))\le \mu\lp B_{d}^{V}(x_0, r)\rp.\]

\end{proof}
\begin{rem}
  \rm{In general, we can not expect that $d$ and $d_{l}$ coincide on $V$, as $d$ is not assumed to be a shortest path metric.
  }
\end{rem}

\section{Volume growth criterion}\label{sect-proof-vol}
The proof of Theorem \ref{thm-Folz} is a combination of Lemma \ref{lem-comparison} and Proposition \ref{prop-comparison}.
Let $\g$ be a simple weighted graph  with an adapted metric $d$, and $X$ be the corresponding metric graph as constructed in Section \ref{sect-metric-graph} (with metric $d_{l}$ and measure $\hat{\mu}$). By Lemma \ref{lem-comparison}, the metric $d_{l}\ge d$ when restricted on $V$. In Theorem \ref{thm-Folz}, we assume the finiteness of all metric balls in $(V, d)$, whence metric balls in $(V, d_{l})$ are finite as well. So the results about metric graphs in subsection 0.5 apply.
In particular, we need that $d_{l}=\rho$, the intrinsic metric on $X$. The basic Assumption \ref{assm-basic} of Sturm is then fulfilled.

In view of the volume growth comparison in Lemma \ref{lem-comparison}, the volume growth of the metric graph $X$ satisfies
   \[\int^{\infty}\frac{rdr}{\log \lp\hat{\mu}\lp B_{d}(x_0, r)\rp\rp}=\infty,\]
   for any $x_0\in V$ if so does the weighted graph $\g$. Sturm's volume growth criterion shows stochastic completeness of $X$. Then the stochastic completeness of $\g$ follows from Proposition \ref{prop-comparison}.

In this section, we give our first analytic proof of Proposition \ref{prop-comparison}.

\begin{proof}[Proof of Proposition \ref{prop-comparison}]

We assume that $\g$ is stochastically incomplete. By Theorem \ref{thm-weak-OY-graph}, there exists a non-negative, nonzero, bounded function $u$ on $V$ such that
\[\Delta u (x)=\lap(u(x)-u(y))\le-1,\] 
for all $x\in \Omega_{1}=\{y\in V: u(y)>u^*-1\}$. 
Without loss of generality, we can also assume that $u^*=\sup_V u=2$.


  Claim: The metric graph $X$ with the Dirichlet form $(\eee, \fff)$ as constructed in subsection 0.5 is stochastically incomplete.

   It suffices to construct a WOYMP-violating function $v$ on $X$ according to Proposition \ref{prop-weak-OY}. It is constructed in several steps:
   \begin{enumerate}
     \item for any $x\in V$, we let $v(x)=u(x)$;
     \item for an edge $e=(x,y)\in E_+$ such that $e$ has at least one vertex in $\Omega_{1}$, solve the simple equation
     $v''=1$ on $(0, l(e))$ with boundary condition $v(0)=u(x), v(l(e))=u(y)$;
     \item for an edge $e=(x,y)\in E_+$ that has no vertex in $\Omega_{1}$, we define $v$ to be linear on $[0, l(e)]$ with
     $v(0)=u(x), v(l(e))=u(y)$.
   \end{enumerate}
We list the solution to the simple boundary value problem on a marked interval $I(e)$:
\[v(t)=\frac{1}{2}t^2+ \lp \frac{v(l(e))-v(0)}{l(e)}-\frac{1}{2}l(e)\rp t+ v(0).\]
What we really need is the derivatives of $v$ at the boundary points:
\[v'(0)=\frac{v(l(e))-v(0)}{l(e)}-\frac{1}{2}l(e), v'(l(e))=\frac{v(l(e))-v(0)}{l(e)}+\frac{1}{2}l(e).\]
Another simple observation is that the maximum of $v\vert_{I(e)}$ is always achieved at the boundary points for each $e\in E_+$.
Note that we are solving a Dirichlet to Neumann problem, although a simple one.

It is easy to see that the function $v$ is locally Lipschitz and bounded with $v^*=\sup_X v=\sup_V u=2$, thus in $C_{\func{Lip}, \func{loc}} (X)\cap L^{\infty}\subseteq \fff_{b, \func{loc}}$. Denote \[\hat{\Omega}_{1}=\{x\in X: v(x)>v^*-1=1\}.\]
Note that $\hat{\Omega}_{1}\cap V=\Omega_{1}$.

   Fix an arbitrary non-negative $\phi\in \fff_{b, c}$ with $\func{supp}\phi \subset \overline{\hat{\Omega}_{1}}$. We always choose the continuous representative of $\phi$, whence $\phi =0$ outside $\hat{\Omega}_{1}$. 
 Let $A\subset E_+$ be such that the corresponding marked intervals have nonempty intersection with $\hat{\Omega}_{1}$. We have that
   \begin{align}
\nonumber\eee(v, \phi)&=\sum_{e\in A}p(e)\int_{I(e)} v' \phi' dm(e)\\
\nonumber&=\sum_{e\in A}p(e)\int_{0}^{l(e)}  \lp \int_{0}^{t} v''(s)ds+ v\vert_{I(e)}'(0)\rp\phi'(t) dt\\
\nonumber&=\sum_{e\in A}p(e)\int_{0}^{l(e)}  v''(s) \lp \int_{s}^{l(e)} \phi'(t)dt\rp ds
+\sum_{e\in A}p(e) v\vert_{I(e)}'(0)(\phi(l(e))-\phi(0))\\
&=-\sum_{e\in A}p(e)\int_{0}^{l(e)}  v''(s) \phi(s) ds\\&+\sum_{e\in A}p(e) \lp v\vert_{I(e)}'(l(e))\phi(l(e))- v\vert_{I(e)}'(0)\phi(0)\rp.
\end{align}
The term in (3.1) equals to \[-\int_{\hat{\Omega}_{1}}\phi d\hat{\mu},\]
noting that $v''=1$ on the edges in $A$, and $p=q$.
The term in (3.2) can be rewritten as a sum over vertices since
\begin{align}\label{eq-edge-term}
&v\vert_{I(e)}'(l(e))\phi(l(e))- v\vert_{I(e)}'(0)\phi(0)\\ \nonumber =&\lp \frac{u(t(e))-u(s(e))}{l(e)}+\frac{1}{2}l(e) \rp\phi(t(e))\\\nonumber
-&\lp \frac{u(t(e))-u(s(e))}{l(e)} -\frac{1}{2}l(e)\rp\phi(s(e)).
\end{align}
It is direct to see that for each $x$ with either $(x,y)\in A$ or $(y,x)\in A$, the contribution of (\ref{eq-edge-term}) at $x$ is
\[\lp \frac{u(x)-u(y)}{d(x,y)} +\frac{1}{2}d(x,y)\rp\phi(x).\]

Noting that $\phi=0$ outside $\hat{\Omega}_{1}$, we obtain
\begin{align*}
&\sum_{e\in A}p(e) \lp v\vert_{I(e)}'(l(e))\phi(l(e))- v\vert_{I(e)}'(0)\phi(0)\rp\\
=&\sum_{x\in \Omega_{1}} \sum_{y, y\sim x}\omega(x,y)d(x,y)\lp \frac{u(x)-u(y)}{d(x,y)} +\frac{1}{2}d(x,y)\rp\phi(x)\\
=&\sum_{x\in \Omega_{1}} \phi(x)\lp\sum_{y, y\sim x}\omega(x,y)(u(x)-u(y)) +\frac{1}{2}\sum_{y, y\sim x}\omega(x,y)d^2(x,y)\rp\\
\le&\sum_{x\in \Omega_{1}} \phi(x)\mu(x)\lp -1+\frac{1}{2} \rp\le0.
\end{align*}

In summary, we have that
\[\eee(v, \phi)\le-\int_{\hat{\Omega}_{1}}\phi d\hat{\mu}.\]
The claim follows.

\end{proof}
\begin{rem}\rm{ From the proof we can see the flexibility of Proposition \ref{prop-weak-OY}. It is easy to work with, thanks to the weak formulation.

}
\end{rem}

\section{Another proof of Proposition \ref{prop-comparison}}\label{sect-another-proof}
In this section, we give another proof of Proposition \ref{prop-comparison}. As before, $(V, \omega,\mu)$ is a simple weighted graph with an adapted metric $d$ as in Definition \ref{defi-adapted-distance-new}. 

An immediate consequence of Theorem \ref{thm-weak-OY-graph} is the following stability type result for a weighted graph.
\begin{lem}
  \label{lem-change-mu} Let $\g$ be a simple weighted graph with an adapted metric $d$. Assume that $\g$ is stochastically incomplete.
  Let $\nu: V\rightarrow (0,\infty)$ be such that $\nu(x)\le \mu(x)$ for each $x\in V$.
  Consider the weighted graph $(V, \omega,\nu)$. We have that $(V, \omega,\nu)$ is stochastically incomplete.
 A special case is that we define $\nu$ by $\nu(x)=\sum_{y}\omega(x,y)d^2(x,y)$.
\end{lem}
\begin{proof}
  By Theorem \ref{thm-weak-OY-graph}, we let $u$ be a non-negative WOYMP-violating function on $\g$.
  Denote the formal Laplacian on $(V, \omega,\nu)$ by $\Delta'$. For each $x\in \Omega_{\alpha}$,
  we have
  \[\Delta' u(x)=\frac{1}{\nu(x)}\sum_{y}\omega(x,y)\lp u(x)-u(y)\rp=\frac{\mu(x)}{\nu(x)}\Delta u(x)\le -\frac{\mu(x)}{\nu(x)}\alpha\le-\alpha.\]
  Thus $u$ is also a WOYMP-violating function on $(V, \omega,\nu)$.
\end{proof}
\begin{rem}\label{rem-special-mu}
  \rm{Note that $d$ remains an adapted metric on $(V, \omega,\nu)$, and the volume growth in $(V, d,\nu)$ is smaller than that of $(V, d,\mu)$. Therefore, by the above lemma, for the volume growth criterion of stochastic completeness for a weighted graph $\g$,
  we can focus our attention on the special case $\mu(x)=\sum_{y}\omega(x,y)d^2(x,y)$ for each $x\in V$. 
  }
\end{rem}
\begin{proof}
  [Another proof of Proposition \ref{prop-comparison}]
  Let $\g$ be a simple weighted graph with an adapted metric $d$. By Lemma \ref{lem-change-mu} and Remark \ref{rem-special-mu}, we can assume that $\mu(x)=\sum_{y}\omega(x,y)d^2(x,y)$ for each $x\in V$ without loss of generality. Let $X$ be the corresponding metric graph with metric $d_{l}$ and measure $\hat{\mu}$ as in Section \ref{sect-metric-graph}. To avoid confusion, in this section, we denote the minimal Dirichlet form on the weighted graph by $(\eee, \fff)$, and the one on $X$ by $(\hat{\eee}, \hat{\fff})$. Similar notations apply for a quantity on the weighted graph with a corresponding one on the associated metric graph.


  Let $\{\hat{v}_n\}\subset \hat{\fff}$ be a sequence of functions
  satisfying
\[0\le \hat{v}_n\le 1, \lim_{n\rightarrow \infty}\hat{v}_n =1 ~~~~\hat{\mu}\text{-a.e.}\] 
such that
\[\lim_{n\rightarrow \infty}\hat{\eee}(\hat{v}_n, \hat{w})=0\]
holds for any $\hat{w}\in \hat{\fff}\cap L^1(X, \hat{\mu})$,
  as given by Theorem \ref{thm-fot}. We always work with the unique continuous representative for each $\hat{v}_n$. By continuity, we see that $0\le \hat{v}_n\le 1$ on $X$ for each $n$.
  Define a sequence of functions $\{v_n\}$ on $V$ by $v_n= \hat{v}_n\vert_V$ for each $n$. It is direct to see that $0\le v_n\le 1$ on $V$ for each $n$ by continuity of $\hat{v}_n$.

  Claim 1: The sequence $\{v_n\}\subset \fff$.

  By the essential self-adjointnees result in \cite{HKMW} (Theorem 1), $\fff=\fff_{\max}$. Hence it suffices to show that for each $n$ 
  \[\sum_{x\in V}v_n^2(x)\mu(x)+\frac{1}{2}\sum_{x\in V}\sum_{y\in V}\omega(x,y)\lp v_n(x)-v_n(y)\rp^2<\infty.\]

We need the following Sobolev embedding $W^{1,2}((0, l))\subset C([0,l])$ with optimal constant (\cite{Rich-sobolev}, \cite{WKNTY}), 
\[\lp\sup_{t\in [0, l]}\vert u(t) \vert\rp^2\le \coth (l) \int_{0}^{l}\lp u^2(t)+(u'(t))^2 \rp dt.\]
It follows that for each marked interval $I(e)$ in $X$ with $e=(x,y)\in E_+$, for each $u\in W^{1,2}(X)\cap C(X)$,
\begin{equation}
\label{eq-sup-sob}
\parallel u\vert_{I(e)} \parallel_{\sup}^2\le \frac{\coth \lp d(x,y)\rp}{\omega(x,y)d(x,y)}\parallel  u\vert_{I(e)}\parallel_{W^{1,2}(I(e))}^2.
\end{equation}
Then we have that for each $u\in W^{1,2}(X)\cap C(X)$,
\begin{align*}
\sum_{x\in V}u^2(x)\mu(x)&=\sum_{e= (x,y)\in E_+} \omega(x,y)d^2(x,y)\lp u^2(x)+ u^2(y) \rp\\
&\le2\sum_{e= (x,y)\in E_+}d(x,y) \coth \lp d(x,y)\rp\parallel  u\parallel_{W^{1,2}(I(e))}^2\\
&\le C \parallel u \parallel_{W^{1,2}(X)}^2,
\end{align*}
where $C=2\sup_{t\in (0,c_0]}t\coth(t)>0$ is a constant.
And for each $u\in W^{1,2}(X)\cap C(X)$,
\begin{align*}
\frac{1}{2}\sum_{x\in V}\sum_{y\in V}\omega(x,y)\lp u(x)-u(y)\rp^2&=\sum_{e= (x,y)\in E_+}\omega(x,y)\lp \int_{0}^{l(e)}u'(t)dt\rp^2\\
&\le\sum_{e= (x,y)\in E_+}\omega(x,y)d(x,y)\lp \int_{0}^{l(e)}\lp u'(t)\rp^2 dt\rp\\
&\le \parallel u \parallel_{W^{1,2}(X)}^2.
\end{align*}
Claim 1 follows.

Claim 2: For each $x\in V$, $\lim_{n\rightarrow \infty}v_n(x)=1$.

Fix $x\in V$ and recall that $E_x\subset E_+$ is the set of marked intervals with a vertex being $x$. By assumption, since $\lim_{n\rightarrow\infty} \hat{v}_n=1$, $\hat{\mu}$-a.e., we can choose some $y_e\in (0, l(e))$ for each $e\in E_x$ such that $\lim_{n\rightarrow\infty} \hat{v}_n (y_e)=1$ for each $e$.
Define $\hat{w}$ on $\cup_{e\in E_x} I(e)$ by $\hat{w}(y)=\frac{1}{d(x,y_e)}(d(x,y_e)-d(x,y))_+$  for $y\in I(e)$ and extend it by $0$ outside. It is easy to see that $\hat{w}$ is compactly supported, Lipschitz and thus $\hat{w}\in L^1(X, \hat{\mu})\cap\hat{\fff}$.
Then we have that
\begin{align*}
0=\lim_{n\rightarrow\infty}\hat{\eee}(\hat{v}_n, \hat{w})&=\lim_{n\rightarrow\infty}\sum_{e\in E_x}\omega(e)l(e)\int_{I(e)}\hat{v}_n'(t)\hat{w}'(t) dt\\
&=\lim_{n\rightarrow\infty}\sum_{e\in E_x}\omega(e)l(e)\frac{1}{d(x,y_e)}(\hat{v}_n(x)-\hat{v}_n(y_e)),
\end{align*}
whence $\lim_{n\rightarrow \infty}v_n(x)=\lim_{n\rightarrow \infty}\hat{v}_n(x)=1$, as all sums in the above have only finitely many terms.
Claim 2 follows.

Now for any function $w$ on $V$ with $w\in L^1(V, \mu)\cap \fff$, we can extend it to a function $\hat{w}$ on $X$ by linear interpolation with $\hat{w}\vert_V=w$.

Claim 3: The function $\hat{w}\in L^1(X, \hat{\mu})\cap\hat\fff$.

Let $\{w_n\}\subset C_c(V)$ be a sequence converging to $w$ in the $\eee_1$ norm. Let $\hat{w}_n$ be the extension of $w_n$ by linear interpolation in the same way as $\hat{w}$.

On each edge $I(e)$ with $e=(x,y)\in E_+$, we have that $\hat{w}$ (as well as $\hat{w}_n$) is a linear function. Thus for each $e=(x,y)\in E_+$, we have
\[\omega(x,y)d(x,y)\int_{I(e)}\lp \hat{w}'(t)\rp^2 dt=\omega(x,y)\lp \hat{w}(x)-\hat{w}(y)\rp^2,\]
and
\begin{align*}
\omega(x,y)d(x,y)\int_{I(e)}\hat{w}^2(t) dt&=\frac{1}{3}\omega(x,y)d^2(x,y)\lp w^2(x)+w(x)w(y)+w^2(y)\rp\\
&\le \frac{1}{2}\omega(x,y)d^2(x,y)\lp w^2(x)+w^2(y)\rp,
\end{align*}
whence $\hat{\eee}_1(\hat{w},\hat{w})\le\eee_1(w, w)$. The same estimates clearly hold with $w, \hat{w}$ replaced by $w-w_n, \hat{w}-\hat{w}_n$ respectively for each $n$, that is, $\hat{\eee}_1(\hat{w}-\hat{w}_n,\hat{w}-\hat{w}_n)\le\eee_1(w-w_n, w-w_n)$ for each $n$. It follows that $\hat{w}_n\rightarrow \hat{w}$ with respect to $\hat{\eee}_1$-norm and $\hat{w}\in \hat{\fff}$.

To show that $\hat{w}\in L^1(X, \hat{\mu})$, we need another elementary calculation for each $e=(x,y)\in E_+$:
\[\omega(x,y)d(x,y)\int_{I(e)}\vert \hat{w}(t)\vert dt=\frac{1}{2}\omega(x,y)d^2(x,y)\lp \vert w(x)\vert +\vert w(y)\vert\rp,\]
by properties of linear functions. It follows that $\parallel \hat{w}\parallel_{L^1(X, \hat{\mu})}=\frac{1}{2}\parallel w\parallel_{L^1(V, \mu)}$.

Now we finish the proof by noting that
\begin{align*}
\eee(u_n, w)&=\sum_{e=(x,y)\in E_+}\omega(x,y)\lp v_n(x)-v_n(y)\rp \lp w(x)-w(y)\rp\\
&=\sum_{e=(x,y)\in E_+}\omega(x,y)d(x,y)\int_{I(e)}\hat{v}_n'(t) \hat{w}'(t)dt\\&=\hat{\eee}(\hat{v}_n, \hat{w})\rightarrow 0,
\end{align*}
as $n\rightarrow\infty$.
\end{proof}

\section*{Acknowledgement}
The author would like to thank Prof. Grigor'yan for stimulating discussions. The proof of Lemma \ref{lem-technical-oy} is suggested by him. We are also grateful to Matthias Keller, Jun Masamune, Marcel Schmidt and Radek Wojciechowski for helpful comments. 

Part of this work was done when the author was visiting the group of Prof. Lenz at Friedrich-Schiller-Universit\"{a}t Jena and Research group J\"{u}rgen Jost at the Max Planck Institute for Mathematics in the Sciences (MIS) in Leipzig. We would like to thank both institutions for their hospitality.

Special thanks to Sebastian Haeseler, who explained many questions about metric graphs to the author.

\bibliographystyle{amsplain}
\bibliography{ref}

\def\cprime{$'$}
\providecommand{\bysame}{\leavevmode\hbox to3em{\hrulefill}\thinspace}
\providecommand{\MR}{\relax\ifhmode\unskip\space\fi MR }
\providecommand{\MRhref}[2]{%
  \href{http://www.ams.org/mathscinet-getitem?mr=#1}{#2}
}
\providecommand{\href}[2]{#2}
\begin{thebibliography}{10}

\bibitem{Bar-Bessa}
C.~B{\"a}r and G.~P. Bessa, \emph{Stochastic completeness and volume growth},
  Proc. Amer. Math. Soc. \textbf{138} (2010), no.~7, 2629--2640.

\bibitem{BirMos91}
M.~Biroli and U.~Mosco, \emph{Formes de {D}irichlet et estimations
  structurelles dans les milieux discontinus}, C. R. Acad. Sci. Paris S\'er. I
  Math. \textbf{313} (1991), no.~9, 593--598.

\bibitem{BirMos95}
\bysame, \emph{A {S}aint-{V}enant type principle for {D}irichlet forms on
  discontinuous media}, Ann. Mat. Pura Appl. (4) \textbf{169} (1995), 125--181.

\bibitem{BREbook}
H.~Brezis, \emph{Functional analysis, {S}obolev spaces and partial differential
  equations}, Universitext, Springer, New York, 2011.

\bibitem{Chung}
K.L. Chung, \emph{Markov chains with stationary transition probabilities},
  Second edition. Die Grundlehren der mathematischen Wissenschaften, Band 104,
  Springer-Verlag New York, Inc., New York, 1967.

\bibitem{CourFriedLewy}
R.~Courant, K.~Friedrichs, and H.~Lewy, \emph{\"{U}ber die partiellen
  {D}ifferenzengleichungen der mathematischen {P}hysik}, Math. Ann.
  \textbf{100} (1928), no.~1, 32--74.

\bibitem{DAV89}
E.~B. Davies, \emph{Heat kernels and spectral theory}, Cambridge Univ. Press,
  1989.

\bibitem{DAV92}
\bysame, \emph{Heat kernel bounds, conservation of probability and the {F}eller
  property}, J. D' Analyse Math. \textbf{58} (1992), 99--119.

\bibitem{CdVBook98}
Y.~Colin de~Verdi{\`e}re, \emph{Spectres de graphes}, Cours Sp\'ecialis\'es
  [Specialized Courses], vol.~4, Soci\'et\'e Math\'ematique de France, Paris,
  1998.

\bibitem{Dod06}
J.~Dodziuk, \emph{Elliptic operators on infinite graphs}, Analysis, geometry
  and topology of elliptic operators, World Sci. Publ., Hackensack, NJ, 2006,
  pp.~353--368.

\bibitem{Dod-Mat06}
J.~Dodziuk and V.~Mathai, \emph{Kato's inequality and asymptotic spectral
  properties for discrete magnetic {L}aplacians}, The ubiquitous heat kernel,
  Contemp. Math., vol. 398, Amer. Math. Soc., Providence, RI, 2006, pp.~69--81.

\bibitem{Feller56}
W.~Feller, \emph{Boundaries induced by non-negative matrices}, Trans. Amer.
  Math. Soc. \textbf{83} (1956), 19--54.

\bibitem{Feller57}
\bysame, \emph{On boundaries and lateral conditions for the {K}olmogorov
  differential equations}, Ann. of Math. (2) \textbf{65} (1957), 527--570.
  \MR{0090928 (19,892b)}

\bibitem{FEL66}
\bysame, \emph{An introduction to probability theory and its applications},
  John Wiley $\&$ Sons, Ltd., New York-London-Sydney, 1966.

\bibitem{FOLZSC}
M.~Folz, \emph{Volume growth and stochastic completeness of graphs}, to appear
  in Trans. Amer. Math. Soc., arXiv:1201.5908.

\bibitem{Folz}
\bysame, \emph{Gaussian upper bounds for heat kernels of continuous time simple
  random walks}, Elec. J. Prob. \textbf{16} (2011), 1693--1722.

\bibitem{FLW}
R.~Frank, D.~Lenz, and D.~Wingert, \emph{Intrinsic metrics for (non-local)
  symmetric {D}irichlet forms and applications to spectral theory}, preprint,
  arXiv:1012.5050v1.

\bibitem{freedman}
D.~Freedman, \emph{Markov chains}, Holden-Day, San Francisco, Calif., 1971.

\bibitem{FOT}
M.~Fukushima, Oshima Y., and M.~Takeda, \emph{{D}irichlet forms and symmetric
  {M}arkov processes}, Walter de Gruyter, Berlin, 1994.

\bibitem{GRI86-87}
A.~Grigor'yan, \emph{On stochastically complete manifolds}, DAN SSSR
  \textbf{290} (1986), 534--537, in Russian. Engl. transl.: Soviet Math. Dokl.,
  34 (1987) no.2, 310-313.

\bibitem{Gri88-89}
\bysame, \emph{Stochastically complete manifolds and summable harmonic
  functions}, Izv. Akad. Nauk SSSR Ser. Mat. \textbf{52} (1988), no.~5,
  1102--1108, 1120.

\bibitem{GRI-SURVEY}
\bysame, \emph{Analytic and geometric background of recurrence and
  non-explosion of the {B}rownian motion on {R}iemannian manifolds}, Bull.
  Amer. Math. Soc. \textbf{36} (1999), 135--249.

\bibitem{GRIBOOK}
\bysame, \emph{Heat kernel and analysis on manifolds}, vol.~{\bf 47}, AMS-IP
  Studies in Advanced Mathematics, 2009.

\bibitem{GriHu-inventiones}
A.~Grigor{\cprime}yan and J.~Hu, \emph{Off-diagonal upper estimates for the
  heat kernel of the {D}irichlet forms on metric spaces}, Invent. Math.
  \textbf{174} (2008), no.~1, 81--126.

\bibitem{Gri-H-Mas}
A.~Grigor'yan, X.~Huang, and J.~Masamune, \emph{On stochastic completeness for
  nonlocal {D}irichlet forms}, Math. Z. (to appear), DOI:
  10.1007/s00209-011-0911-x.

\bibitem{Hae-heat-kernel}
S.~Haeseler, \emph{Heat kernel estimates and related inequalities on metric
  graphs},  (2011), submitted.

\bibitem{HaeWoj}
S.~Haeseler and R.~Wojciechowski, \emph{Stochastic completeness of metric
  graphs},  (2011), preprint.

\bibitem{HSU89}
E.~P. Hsu, \emph{Heat semigroup on a complete {R}iemannian manifold}, Ann.
  Probab. \textbf{17} (1989), 1248--1254.

\bibitem{Huang-unique}
X.~Huang, \emph{On uniqueness class for a heat equation on graphs}, J. Math.
  Anal. Appl. \textbf{393} (2011), no.~2, 377--388.

\bibitem{HXP-OY}
\bysame, \emph{Stochastic incompleteness for graphs and weak {O}mori-{Y}au
  maximum principle}, J. Math. Anal. Appl. \textbf{379} (2011), no.~2,
  764--782.

\bibitem{HKMW}
X.~Huang, M.~Keller, J.~Masamune, and R.~K. Wojciechowski, \emph{A note on
  self-adjoint extensions of the laplacian on weighted graphs},  (2012),
  preprint, arXiv: 1208.6358v1 [math.FA].

\bibitem{LeJan}
Y.~Le Jan, \emph{Mesures associ\'ees \`a une forme de {D}irichlet.
  {A}pplications}, Bull. Soc. Math. France \textbf{106} (1978), no.~1, 61--112.

\bibitem{Karp-Li}
L.~Karp and P.~Li, \emph{The heat equation on complete {R}iemannian manifolds},
  unpublished manuscript 1983.

\bibitem{KL}
M.~Keller and D.~Lenz, \emph{{D}irichlet forms and stochastic completeness of
  graphs and subgraphs}, J. Reine Angew. Math. (accepted), DOI:
  10.1515/CRELLE.2011.122.

\bibitem{Khas}
R.~Z. Khas'minskii, \emph{Ergodic properties of recurrent diffusion processes
  and stabilization of solutions to the {C}auchy problem for parabolic
  equations}, Theor. Prob. Appl. \textbf{5} (1960), 179--195.

\bibitem{KuchSurveyI}
P.~Kuchment, \emph{Quantum graphs. {I}. {S}ome basic structures}, Waves Random
  Media \textbf{14} (2004), no.~1, S107--S128, Special section on quantum
  graphs.

\bibitem{KuchSurveyII}
\bysame, \emph{Quantum graphs. {II}. {S}ome spectral properties of quantum and
  combinatorial graphs}, J. Phys. A \textbf{38} (2005), no.~22, 4887--4900.

\bibitem{MU}
J.~Masamune and T.~Uemura, \emph{Conservation property of symmetric jump
  processes}, Ann. Inst. Henri. Poincar¡äe Probab. Statist. \textbf{47} (2011),
  no.~3, 650--662.

\bibitem{MUW}
J.~Masamune, T.~Uemura, and J.~Wang, \emph{On the conservativeness and
  recurrence of symmetric jump-diffusions},  (2011), preprint.

\bibitem{Omori}
H.~Omori, \emph{Isometric immersions of {R}iemannian manifolds}, J. Math. Soc.
  Japan \textbf{19} (1967), 205--214.

\bibitem{PRSProc}
S.~Pigola, M.~Rigoli, and A.~G. Setti, \emph{A remark on the maximum principle
  and stochastic completeness}, Proc. Amer. Math. Soc. \textbf{131} (2003),
  no.~4, 1283--1288.

\bibitem{PRS03}
\bysame, \emph{Volume growth, ``a priori'' estimates, and geometric
  applications}, Geom. Funct. Anal. \textbf{13} (2003), no.~6, 1302--1328.

\bibitem{PRSSurvey}
\bysame, \emph{Maximum principles on {R}iemannian manifolds and applications},
  Mem. Amer. Math. Soc. \textbf{174} (2005), no.~822, x+99.

\bibitem{Reuter57}
G.~E.~H. Reuter, \emph{Denumerable {M}arkov processes and the associated
  contraction semigroups on {$l$}}, Acta Math. \textbf{97} (1957), 1--46.

\bibitem{Rich-sobolev}
W.~Richardson, \emph{Steepest descent and the least {$C$} for {S}obolev's
  inequality}, Bull. London Math. Soc. \textbf{18} (1986), no.~5, 478--484.

\bibitem{Shiozawa}
Y.~Shiozawa, \emph{Conservation property of symmetric jump-diffusion
  processes}, submitted.

\bibitem{Sturm94}
K.~T. Sturm, \emph{Analysis on local {D}irichlet spaces. {I}. {R}ecurrence,
  conservativeness and {$L^p$}-{L}iouville properties}, J. Reine Angew. Math.
  \textbf{456} (1994), 173--196.

\bibitem{Takeda89}
T.~Takeda, \emph{On a martingale method for symmetric diffusion processes and
  its applications}, Osaka J. Math. \textbf{26} (1989), no.~3, 605--623.

\bibitem{WKNTY}
K.~Watanabe, Y.~Kametaka, A.~Nagai, K.~Takemura, and H.~Yamagishi, \emph{The
  best constant of {S}obolev inequality on a bounded interval}, J. Math. Anal.
  Appl. \textbf{340} (2008), no.~1, 699--706.

\bibitem{WEBER}
A.~Weber, \emph{Analysis of the laplacian and the heat flow on a locally finite
  graph}, J. of Math. Anal. and App. \textbf{370} (2010), 146--158.

\bibitem{WOJ-Indiana}
R.~K. Wojciechowski, \emph{Heat kernel and essential spectrum of infinite
  graphs}, Indiana Univ. Math. J. \textbf{58} (2009), no.~3, 1419--1441.

\bibitem{WOJ-Survey}
\bysame, \emph{Stochastically incomplete manifolds and graphs}, Progress in
  Probability \textbf{64} (2011), 163--179.

\bibitem{Yau75}
S.~T. Yau, \emph{Harmonic functions on complete {R}iemannian manifolds}, Comm.
  Pure Appl. Math. \textbf{28} (1975), 201--228.

\end{thebibliography}

\end{document}